\documentclass[11pt,twoside]{article}
\usepackage{subcaption}
\usepackage{graphicx}

\usepackage[ocgcolorlinks, linkcolor=red]{hyperref}
\usepackage{bm}
\usepackage{bbm}
\usepackage{url}
\usepackage{mwe}
\usepackage{epsfig, graphics, graphicx}
\usepackage{subcaption, comment, bm}
\usepackage[percent]{overpic}
\usepackage{float}
\usepackage{amssymb,bm}
\usepackage{amsthm}
\usepackage{color}
\usepackage[]{amsmath}
\usepackage[]{amsfonts}
\usepackage[]{fancyhdr}
\usepackage[]{graphicx,wrapfig}
\usepackage{enumitem}
\usepackage[utf8]{inputenc}
\usepackage[T1]{fontenc}
\usepackage{mathtools}
\usepackage{array}
\usepackage{comment}
\usepackage{amsmath}
\usepackage{tikz}
\usepackage{epigraph}
\usepackage{lipsum}
\allowdisplaybreaks

\definecolor{titlepagecolor}{cmyk}{1,.60,0,.40}

\DeclareFixedFont{\titlefont}{T1}{ppl}{b}{it}{0.5in}

\def\th@plain{%
  \thm@notefont{}
  \itshape 
}
\def\th@definition{%
  \thm@notefont{}
  \normalfont 
}
\makeatother

\usepackage{srcltx} 
\usepackage{amscd}
\usepackage{amssymb}

\usepackage{amsmath}
\usepackage{pb-diagram}
\usepackage{graphicx}

\usepackage{array}
\usepackage[all]{xy}
\xyoption{matrix}
\xyoption{arrow}
\usepackage{pdfsync}
\usepackage{physics}
\usepackage{enumitem}
 
\usepackage{thmtools}
\usepackage{amssymb}

\DeclareUnicodeCharacter{2212}{-,+}
\usepackage[none]{hyphenat}

\theoremstyle{definition}
\newtheorem{definition}{Definition}[section]
\newtheorem{theorem}[definition]{Theorem}

\newtheorem{prop}[definition]{Proposition}

\newtheorem{remark}[definition]{Remark}

\newcommand{\Ac}{\mathcal{A}}
\newcommand{\Bc}{\mathcal{B}}
\newcommand{\Cc}{\mathcal{C}}

\newcommand{\Fc}{\mathcal{F}}
\newcommand{\Gc}{\mathcal{G}}
\newcommand{\Hc}{\mathcal{H}}
\newcommand{\Ic}{\mathcal{I}}
\newcommand{\Jc}{\mathcal{J}}
\newcommand{\Kc}{\mathcal{K}}
\newcommand{\Lc}{\mathcal{L}}
\newcommand{\Mc}{\mathcal{M}}

\newcommand{\Pc}{\mathcal{P}}

\newcommand{\Rc}{\mathcal{R}}

\newcommand{\Tc}{\mathcal{T}}

\newcommand{\vx}{\textbf{\textit{x}}}

\newcommand{\vw}{\textbf{\textit{w}}}
\newcommand{\vv}{\textbf{\textit{v}}}

\newcommand{\vf}{\textbf{\textit{f}}}

\newcommand{\vg}{\textbf{\textit{g}}}

\newcommand{\DR}{\mathbb{D}_R}
\newcommand{\vuu}{\textbf{\textit{u}}}

\title{\vspace{-1cm} \hspace{1cm}Tensor tomography for a set of generalized V-line transforms in $\mathbb{R}^2$}

\author{Rahul Bhardwaj\thanks{Department of Mathematics, Indian Institute of Technology, Ropar, Punjab-140001, India. \url{rahul.24maz0014@iitrpr.ac.in}}}

\usepackage{fancyhdr}
\usepackage{lipsum}

\newcommand\shorttitle{generalized broken ray transforms in $\mathbb{R}^{2}$}
\newcommand\authors{Rahul Bhardwaj}

\begin{document}
\maketitle
\begin{abstract}
 We study a set of generalized V-line transforms, namely longitudinal, mixed, and transverse V-line transforms, of a symmetric $m$-tensor field in $\mathbb{R}^2$. The goal of this article is to recover a symmetric $m$-tensor field $\vf$ supported in a disk $\DR$, with radius $R$ and centered at the origin, by a combination of the aforementioned generalized V-line transforms, using two different techniques for different
sets of data. 

\end{abstract}
\vspace{2mm}

\section{Introduction }
 Several imaging modalities in computerized tomography are based on the mathematical model, which uses the regular ray transform and the broken ray/V-line transform, see \cite{Natterer_2001,F.Natterer_2001,amb-book} and references therein. Many authors have studied the inverse problems of recovering the scalar function (0-tensor), vector field (1-tensor), $2$-tensor field, and $m$-tensor field in various settings from different combinations of its integral transform such as longitudinal, transverse, mixed, and momentum ray transforms. There are several results available related to inversion of the regular ray transforms (i.e. integrates along a straight line), for instance, see \cite{Denisyuk_1994, Denisyuk_2006, Derevtsov_2015, Katsevich_2006, Mishra_2020, Rohit_Suman_2021, Monard1,  Palamodov_2009,  Sharafutdinov_Book} and references therein. The concept of the V-line transform appears in the model of photon transfer through an object, assumes that photons scatter no more than once, known as single-scattering tomography, see \cite{florescu2009single}. In the last few years, there are various articles focused on studying the broken ray/V-line transforms, which map a function to its integral over the V-shaped lines.  This transform is a generalization of the regular ray transform. Many applications of these transforms appear in the imaging sciences viz, emission tomography \cite{Basko1996AnalyticalRF,Basko1998ApplicationOS}, single scattering optical tomography, see \cite{Florescu2008SinglescatteringOT,Florescu2010InversionFF} and references therein. Now we present a list of the earlier works related to the V-line transforms. The authors in \cite{T.Truong_2011} have studied the V-line transform with a fixed axis of symmetry, and vertices on a line. In \cite{Haltmeier_2017,Haltmeier_attenuated}, the authors addressed the V-line transform and the attenuated V-line transform for the scalar-valued function in a circular geometry structure, where the symmetry axis is perpendicular to the circle and the vertices lies on a circle. In \cite{Indrani_2024}, authors studied the V-line transform in the same setup and considered the inverse problem to recover the symmetric $m$-tensor field from the various combinations. The other class includes V-lines whose vertices are inside the support of the image function and are significant to single-scattering tomography (e.g. optical tomographic imaging in the mesoscopic regime), see \cite{florescu2009single, florescu2010single, Florescu-Markel-Schotland, walker2019broken} and fluorescence imaging in optical tomography, see \cite{florescu_2018}. Based on this class of V-line, the question of recovering a vector field in $\mathbb{R}^2$ from different combinations of new generalized V-line transforms is investigated by  Gaik et al. in \cite{Gaik_Latifi_Rohit, Gaik_Mohammad_Rohit} and also this work is extended by \cite{Gaik_Rohit_Indrani, Indrani_2024v}, Gaik et al. for the case of symmetric $2$-tensor fields in $\mathbb{R}^2$. The set up where the V-line have vertices on the boundary of the image domain are frequently used in image reconstruction with Compton cameras, see \cite{Terzioglu_2018, basko1997fully, basko1997analytical} and references therein. In this setting, authors in  \cite{Ambarsoumian_2013, Ambarsoumian_2012, Ambartsoumian2013} consider the inverse problem for recovering the scalar-valued function by two different techniques. Extension for vector fields is considered in \cite{bhardwaj_2024}, Bhardwaj et al. for the case of the vector field. Also, the broken ray/V-line transforms have been studied in the Riemannian setting  \cite{Shubham,Ilmavirta,Ilmavirta_Mikko,Mark}, based on  PDEs work \cite{Eskin,Carlos_Mikko}, and in the set-up of microlocal analysis by Sherson, see \cite{ Sherson}. Please refer to a recent book by Ambartsoumian \cite{amb-book} for a comprehensive discussion of these generalized operators and an overview of the literature. 
  
Motivated by these works \cite{Ambarsoumian_2012,Ambarsoumian_2013,bhardwaj_2024}, we consider the inverse problem of recovering a symmetric $m$-tensor field supported on the disk $\DR$, using a set of generalized V-line transforms, namely longitudinal, transverse, and mixed V-line transforms. We present two different types of inversion formulas, one for recovering the symmetric $m$-tensor field and other for $2$-tensor field in $\mathbb{R}^2$ from the knowledge of its V-line transforms. Here, we examine two distinct subsets of the previously described broken lines in $\mathbb{R}^2$ to recover the symmetric $m$-tensor field.  In this article, geometrically, we consider the broken rays/V-line that emanate from the boundary of a disk $\DR$ and break into another ray under a fixed angle after travelling a certain distance along the diameter. In this article, we used the approaches given by \cite{Ambarsoumian_2013, Ambarsoumian_2012}, which is done for the scalar case, when the sources and receivers are positioned on the boundary of the disk. We considered a similar set up for symmetric $m$-tensor fields and proved invertibility of the introduced transforms. 

The rest of the article is organized as follows: In Section \ref{sec:Definition and notations}, we introduce the necessary notations and definitions of the integral ray transform. In Section \ref{sec:main results}, we present our main findings of this article. The detailed proof of our main results is provided in Sections \ref{sec:full recovery th} and \ref{sec:partial data case}. Finally, we finish this article with acknowledgements in Section \ref{sec:acknowledgements}. 
 \section{Definitions, notations and statement of main results}\label{sec:Definition and notations}
 \subsection{Definitions and Notations}In this section, we introduce some necessary notations and provide definitions that will be needed later. Let $S^{m}(\mathbb{R}^{2})$ be the space of symmetric $m$-tensor fields on $\mathbb{R}^2$. Let $\Omega$ be a bounded domain in $\mathbb{R}^2$ and  $C_0^{\infty}(\Omega; S^{m}(\mathbb{R}^{2}))$  be the space of $S^{m}(\mathbb{R}^{2})$ valued function with each component is smooth and compactly supported in $\Omega$. Throughout this article, we have used bold font letters for vectors and $m$-tensor fields (such as $\vf$, $\vx$, $\vuu$, etc.) and regular font letters for scalars and scalar-valued functions (such as $u_1$, $u_2$, $t$, $f_{i_1i_2\cdots i_m}$, etc). Any $\vf \in C_0^{\infty}(\Omega; S^{m}(\mathbb{R}^{2}))$ is defined as
 \begin{align}\label{Einstein summation}
     \vf(\vx) := f_{i_1i_2\cdots i_m}(\vx)\,dx^{i_1}\,dx^{i_2}\,\cdots\,dx^{i_m}
 \end{align}
 where for each $i_1,i_2,\cdots,i_m = 1, 2$, $f_{i_1i_2\cdots i_m}$ are smooth functions with compact support and also symmetric with their indices. Note that in equation \eqref{Einstein summation} we have used the \textquotedblleft{\textbf{Einstein summation rule}}'' which allows us to assume the summation from $1$ to $2$ over every repeated indices and the same will be followed through out this article. The standard dot product $\displaystyle \left\langle \cdot,\cdot\right\rangle$ on $ S^{m}(\mathbb{R}^{2})$  is given by  \begin{align*}
    \left\langle \vf,\vg \right\rangle = f_{i_1i_2\cdots i_m}g^{i_1i_2\cdots i_m} 
\end{align*}
and we denote by $|\cdot|$  the corresponding norm.  Let $\vuu = (u^1, u^2) \in \mathbb{R}^2$, then $\vuu^m \in S^{m}(\mathbb{R}^2)$, where $\vuu^m$ stands for the $m^{th}$ symmetric tensor product of $\vuu$ and it is given by
$\vuu^{m} := u^{i_1}u^{i_2}\cdots u^{i_m} $, where $i_1,i_2,\cdots ,i_m = 1, 2$.

Let $\vf$ be a symmetric $m$-tensor field supported in the disk $\DR$ and $\theta \in (0, \pi/2)$ be a fixed scattering angle. A combination of two rays which emanate from the point $\vx_\phi \in \partial \DR$ and moves the distance $d$ along the diameter in the radial direction $\vuu_\phi$, then breaks into another ray under the obtuse angle $\pi - \theta$ and travels in the direction $\vv_\phi$ (see \ref{fig: def of broken ray}), is called the \textbf{broken ray} and is denoted by $BR(\phi, d)$. Mathematically, the  broken ray $BR(\phi, d)$ is described by:
\begin{align}\label{eq:definition of BR(beta,d)}
    BR(\phi,d) =  \left\{\vx_\phi + t \vuu_\phi: 0\leq t \leq d\right\}  \cup \left\{\vx_\phi + d \vuu_\phi + s \vv_\phi: 0 \leq s < \infty\right\}.
\end{align}
 Here $\vx_\phi = (R \cos \phi, R \sin \phi)$, $\vuu_\phi =  -(\cos \phi, \sin \phi)$ and $\vv_\phi = -\left(\cos (\theta + \phi), \sin (\theta +\phi)\right)$, lie on the unit circle $\partial\mathbb{D}_1$ and the vectors  $\vuu_{\phi}^\perp = (\sin \phi, - \cos \phi)$ and $\vv_\phi^{\perp} = \left(\sin (\theta +\phi), -\cos (\theta + \phi)\right)$  are obtained by rotating $\vuu_\phi$ and $\vv_\phi$ in the direction of counterclockwise by an angle of $\pi/2$ respectively.

The unit vectors $\vuu_\phi$ and $\vv_\phi$, which are used in the following definitions, are fixed and uniquely determined by $(\phi, d)$.
\begin{figure}[tbp]
\centering
\begin{subfigure}[t]{0.48\textwidth}
\includegraphics[width=1.06\textwidth]{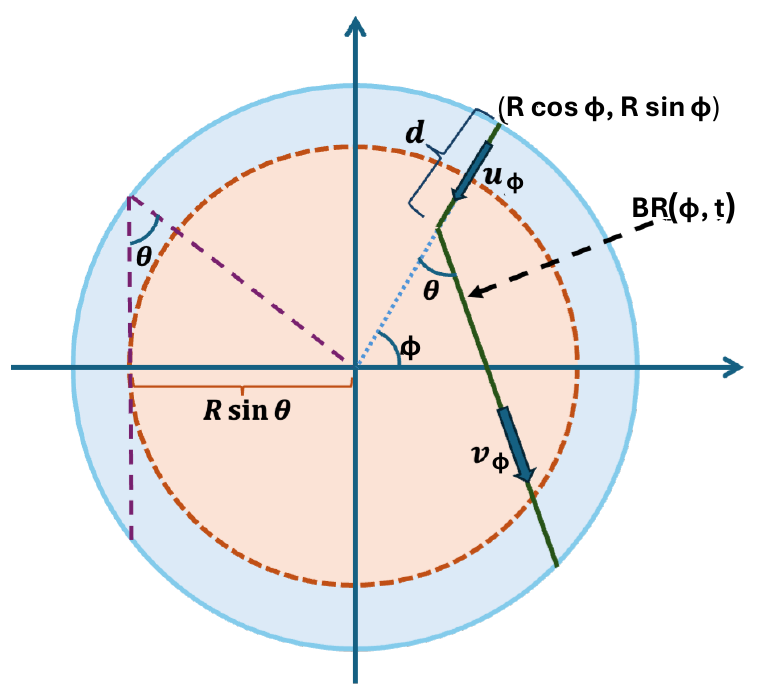}\caption{A sketch of broken line $BR(\phi, d)$.}\label{fig: def of broken ray}
\end{subfigure}
\hfill
\begin{subfigure}[t]{0.48\textwidth}
\centering
\includegraphics[width=1.1\textwidth]{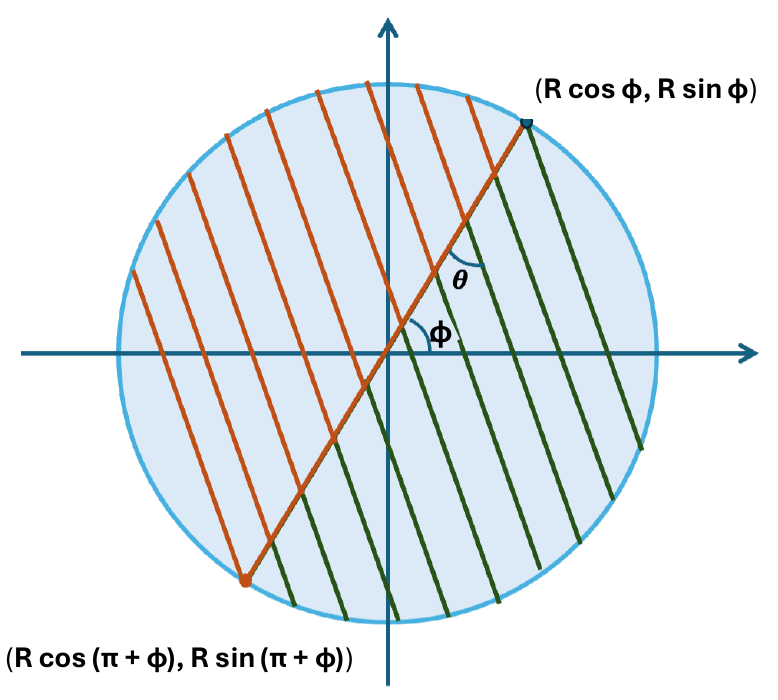}\caption{Broken lines considered to collect data.}\label{fig: data is taken over}
\end{subfigure}
\end{figure}
Now, we introduce the broken ray/V-line transforms, defined as the integrals over the broken ray $BR(\phi,d)$.
\begin{definition}\label{def:longitudinal v-line transform}
For $\vf \in C_0^{\infty}(\DR; S^{m}(\mathbb{R}^{2}))$, the \textbf{longitudinal V-line transform} of $\vf$ is defined by
\begin{equation}\label{eq:def LVT}
\mathcal{L}\vf (\phi,d) := \int_{0}^{d} \left\langle \vf (\vx_\phi +s\vuu_\phi), \vuu_\phi^{m} \right\rangle \,ds + \int_{0}^{\infty} \left\langle\vf(\vx_\phi + d\vuu_\phi + s\vv_\phi), \vv_\phi^{m} \right\rangle\,ds,
\end{equation}
where $\phi \in [0, 2\pi)$ and $ d \in [0,2R]$. 
\end{definition}

\begin{definition}\label{def:transverse v-line transform}
For $\vf \in C_0^{\infty}(\DR; S^{m}(\mathbb{R}^{2}))$, the \textbf{transverse V-line transform} of $\vf$ is defined by
\begin{equation}\label{eq:def TVT}
\mathcal{T}\vf (\phi,d) := \int_{0}^{d} \left\langle\vf (\vx_\phi +s\vuu_\phi),(\vuu_\phi^\perp)^{m} \right\rangle\,ds + \int_{0}^{\infty}  \left\langle \vf(\vx_\phi + d\vuu_\phi + s\vv_\phi),(\vv_\phi^\perp)^{m}\right\rangle\,ds,
\end{equation} 
where $\phi \in [0, 2\pi)$ and $ d \in [0,2R]$. 

\end{definition}

\begin{definition}\label{def:mixed v-line transform}
For $\vf \in C_0^{\infty}(\DR; S^{m}(\mathbb{R}^{2}))$ and $1 \leq k \leq m-1$, the \textbf{mixed V-line transform} of $\vf$ is defined by
\begin{equation}\label{eq:def MVT}
\mathcal{M}^{(k)}\vf (\phi,d) := \int_{0}^{d} \left\langle\vf (\vx_\phi +s\vuu_\phi),(\vuu_\phi^\perp)^{k}\vuu_\phi^{m-k} \right\rangle\,ds + \int_{0}^{\infty}  \left\langle \vf(\vx_\phi + d\vuu_\phi + s\vv_\phi), (\vv_\phi^\perp)^{k}\vv_\phi^{m-k}\right\rangle\,ds,  
\end{equation}
where $(\vuu_\phi^\perp)^{k}\vuu_\phi^{m-k} := (\vuu_\phi^{\perp})^{i_1}(\vuu_\phi^{\perp})^{i_2}\cdots (\vuu_\phi^{\perp})^{i_k}(\vuu_\phi^{\perp})^{i_{k+1}}\cdots \vuu_\phi^{i_m} $, for $i_1,i_2,\cdots ,i_m = 1, 2$, $\phi \in [0, 2\pi)$ and $ d \in [0,2R]$. 
\end{definition}
\begin{remark}
In the definition of the mixed V-line transform, when $k=0$, the transform reduces to the definition of the longitudinal V-line transform, and when $k=m$, it reduces to the transverse V-line transform.
\end{remark}
\noindent Subsequently, we introduce the longitudinal, transverse, and mixed ray transforms, which will be referenced later in this article. These transforms are integrated over a straight line, defining the line \( L(\psi,p):= \{(y_1,y_2) \in \mathbb{R}^2: y_1 \cos{\psi} + y_2 \sin{\psi} = p\} \), for a given angle $\psi \in [0, 2\pi)$ and a real number $p$. This line is positioned at a signed distance $p$ from the origin and is perpendicular to the unit vector $\vw = (\cos{\psi},\sin{\psi})$. The vector $\vw^\perp = (-\sin{\psi},\cos{\psi})$ is also a unit vector in $\mathbb{R}^2$ that is perpendicular to the vector $\vw$.
\begin{definition}\label{def:longitudinal ray transform}
For $\vf \in C_0^{\infty}(\DR; S^{m}(\mathbb{R}^{2}))$, the \textbf{longitudinal ray transform} of $\vf$ is defined by
\begin{equation}\label{eq:def LRT}
\mathcal{I}\vf (\psi,p) = \mathcal{I}\vf (\vw,p) := \int_\mathbb{R} \left\langle\vf (p \vw + s\vw^\perp), (\vw^\perp)^{m} \right\rangle\,ds,  \quad \psi \in [0, 2\pi) \mbox { and } p \in \mathbb{R}.
\end{equation}
\end{definition}

\begin{definition}\label{def:transverse ray transform}
For $\vf \in C_0^{\infty}(\DR; S^{m}(\mathbb{R}^{2}))$, the \textbf{transverse ray transform} of $\vf$ is defined by
\begin{equation}\label{eq:def TRT}
\mathcal{J}\vf (\psi,p) = \mathcal{J}\vf (\vw,p) := \int_\mathbb{R} \left\langle \vf (p \vw + s\vw^\perp), \vw^{m} \right\rangle\,ds, \quad \psi \in [0, 2\pi) \mbox { and } p \in \mathbb{R}.
\end{equation}
\end{definition}  

\begin{definition}\label{def:mixed ray transform}
For $\vf \in C_0^{\infty}(\DR; S^{m}(\mathbb{R}^{2}))$ and $1 \leq k \leq m-1$, the \textbf{mixed ray transform} of $\vf$ is defined by
\begin{equation}\label{eq:def MRT}
\mathcal{K}^{(k)}\vf (\psi,p) = \mathcal{J}\vf (\vw,p) := \int_\mathbb{R} \left\langle \vf (p \vw + s\vw^\perp), \vw^{k}(\vw^\perp)^{m-k}\right\rangle\,ds, 
\end{equation}
where $\vw^{k}(\vw^\perp)^{m-k} := w^{i_1}w^{i_2}\cdots w^{i_k}(w^{\perp})^{i_{k+1}}\cdots (w^{\perp})^{i_{m}} $, for $i_1,i_2,\cdots ,i_m = 1, 2$, $\psi \in [0, 2\pi)$ and $p \in \mathbb{R}$.
\end{definition}
 
\noindent Furthermore, we provide the definition of the Radon transform, which we also require later on:
\begin{definition}\label{def:the Radon transform}
Let $f$ be a scalar function field in  $C_0^\infty(\DR)$. The \textbf{Radon transform} $f$ is defined as follows:
\begin{equation}\label{eq:def Radon transform}
\Rc f (\psi,p) = \Rc f (\vw,p) := \int_\mathbb{R} f (p \vw + s\vw^\perp)\,ds, \quad \psi \in [0, 2\pi) \mbox { and } p \in \mathbb{R}.
\end{equation}
\end{definition}
\subsection{Statement of main results}\label{sec:main results}
In this subsection, we state the main results of this article. 

\begin{theorem}\label{th:full data recovery}
Let $\vf$ be a symmetric $m$-tensor field with each component smooth and compactly supported in a disk $\mathbb{D}_{R\sin{\theta}}$. Then we can recover  $\vf$ explicitly with an inversion formula,  from the knowledge of   $\Lc \vf (\phi, d)$, $\Tc \vf (\phi, d)$ and $\Mc^{(k)} \vf (\phi, d)$ where $1\leq k \leq m-1$, for $d \in [0, 2R]$ and $\phi \in [0, 2\pi)$. 
\end{theorem}
Observe from the \ref{fig: def of broken ray} that we have no data about straight-line transform outside the disk $D_{R\sin \theta}$ hence the support condition on tensor field $\vf$ in Theorem \ref{th:full data recovery}, is imposed  to conclude the recovery of $\vf$.
We will prove Theorem  \ref{th:full data recovery},  by transferring the existing knowledge about the V-line transforms to that of  the knowledge about corresponding straight-line transforms of tensor field in $\mathbb{R}^2$. After that following the approach used in \cite[Section 5.2]{Derevtsov_2015}, we recover the Radon transform of each component of the symmetric tensor field. Finally using the inversion formula for the Radon transform for  functions, we derive the inversion formula for tensor fields. The approach used to prove the above Theorem is motivated  by \cite{Ambarsoumian_2012} and \cite{bhardwaj_2024} where the analogous problem is considered for  the scalar function and the vector field respectively. 
\begin{theorem}\label{th:partial data recovery}
Let $\vf$ be a symmetric $2$-tensor field with each component smooth and compactly supported in a disk $\mathbb{D}_{R}$. Then $\vf$ is uniquely recovered from the knowledge of $\Lc \vf (\phi, d)$, $\Tc \vf (\phi, d)$ and $\Mc \vf (\phi, d)$, where  $d \in [0, R]$ and $\phi \in [0, 2\pi)$. 
\end{theorem}
\noindent  {\textbf{Approach to establish the proof of Theorem \ref{th:partial data recovery}:} We expand the data ($\Lc \vf, \Tc\vf \ \&\ \Mc\vf$) and the unknown $m$-tensor field $\vf$ into their Fourier series and then try to express the Fourier coefficients of $\vf$ in terms of the Fourier coefficients of $\Lc \vf $, $\Tc \vf $ and $\Mc^{(k)} \vf $ where $1\leq k \leq m-1$. Finally using the  Mellin's transform, we prove that the Fourier coefficients of each  components of symmetric $2$-tensor field can be expressed in terms of known data. This will give the required series formula for recovery of the tensor field. This problem has been studied earlier in \cite{Ambarsoumian_2013} for the scalar case and in \cite{bhardwaj_2024} for the case of vector field in $\mathbb{R}^2$. 

\begin{remark}
    In Theorem \ref{th:full data recovery}, there is a restriction on the support of the symmetric $m$-tensor field $\vf$,  which depends on the fixed scattering angle $\theta$. This support condition arises due to the technique we are using to prove the theorem. For the case of $2$-tensor field, the result in  Theorem \ref{th:partial data recovery}, is  stronger  than that of Theorem \ref{th:full data recovery}, since it does not impose restrictions on the support of $\vf$ even after using the less data as compared to Theorem \ref{th:full data recovery}, in the sense that scalar $d$ varies in the half interval $[0, R]$ instead of $[0, 2R]$  in Theorem \ref{th:full data recovery}.
\end{remark}
\section{Proof of Theorem \ref{th:full data recovery} }\label{sec:full recovery th}
In this section, we show that the unknown symmetric $m$-tensor field $\vf$ can be uniquely determined from the knowledge of its  longitudinal, transverse, and mixed V-line transforms. Throughout this section  symmetric $m$-tensor field $\widetilde{\vf}$ denote the extension of  $\vf$  by zero outside of the disk $\mathbb{D}_{R\sin{\theta}}$ and we represent this extended symmetric $m$-tensor field $\widetilde{\vf}$ again by $\vf$.

\begin{proof}
Let us consider
\begin{align}\label{Sepration of Lf}
 \Lc\vf(\phi,d) = \underbrace{\int_{0}^{d}\left\langle \vf (\vx_\phi +s\vuu_\phi), \vuu_\phi^{m} \right\rangle \,ds}_{A_1} + \underbrace{\int_{0}^{\infty} \left\langle \vf (\vx_\phi + d\vuu_\phi + s\vuu_\phi), \vv_\phi^{m} \right\rangle \,ds,}_{A_2}\\
  \mathcal{T}\vf (\phi,d) = \int_{0}^{d} \left\langle\vf (\vx_\phi +s\vuu_\phi),(\vuu_\phi^\perp)^{m} \right\rangle\,ds + \int_{0}^{\infty}  \left\langle \vf(\vx_\phi + d\vuu_\phi + s\vv_\phi),(\vv_\phi^\perp)^{m}\right\rangle\,ds
  \end{align}
\text{and for} $1\leq k\leq m-1 $
 \begin{align}
     \mathcal{M}^{(k)}\vf (\phi,d) = \int_{0}^{d} \left\langle\vf (\vx_\phi +s\vuu_\phi),(\vuu_\phi^\perp)^{k}\vuu_\phi^{m-k} \right\rangle\,ds + \int_{0}^{\infty}  \left\langle \vf(\vx_\phi + d\vuu_\phi + s\vv_\phi), (\vv_\phi^\perp)^{k}\vv_\phi^{m-k}\right\rangle\,ds.  
 \end{align}
 
We start by observing that if we take $d =2R$, and using the fact that $\vf$ is zero outside of the disk $\DR$, then we have 
\begin{align}\label{eq:relation between VLT and regular RT}
    \Lc \vf(\phi, 2R) =  \Ic \vf(\phi + \pi/2, 0),\\
    \Tc \vf(\phi, 2R) = (-1)^{m} \Jc \vf(\phi + \pi/2, 0),\\
    \Mc^{(k)} \vf(\phi, 2R) =  (-1)^{k}\Kc^{(k)} \vf(\phi + \pi/2, 0).
\end{align}
Also
\begin{align}\label{Separation of Lf with d= 2R-d}
\Lc\vf(\phi + \pi,2R-d) &= \underbrace{\int_{0}^{2R - d}\left\langle \vf (\vx_{\phi + \pi} +s\vuu_{\phi + \pi}), \vuu_{\phi + \pi}^{m} \right\rangle \,ds}_{B_1} \nonumber \\
&\qquad + \underbrace{\int_{0}^{\infty}  \left\langle\vf(\vx_{\phi + \pi}  + (2R - d)\vuu_{\phi + \pi}  + s\vv_{\phi + \pi}),\vv_{\phi + \pi}^m\right\rangle \,ds}_{B_2}.
\end{align}
 Next we simplify the integrals given by $B_1$ and $B_2$. 
\begin{align*}
    B_1 &= \int_{0}^{2R - d}\left\langle \vf (\vx_{\phi + \pi} +s\vuu_{\phi + \pi}), \vuu_{\phi + \pi}^{m} \right\rangle \,ds\\
    &= \int_{0}^{2R - d}\left\langle\vf (\vx_{\phi + \pi} +2R \vuu_{\phi + \pi} - 2R \vuu_{\phi + \pi}  + s\vuu_{\phi + \pi} ), \vuu_{\phi + \pi}^{m}\right\rangle\,ds\\
     &= \int_{0}^{2R - d}\left\langle\vf (\vx_{\phi} + (s - 2R) \vuu_{\phi + \pi}), \vuu_{\phi + \pi}^{m}\right\rangle\,ds, \quad \mbox{ since } \vx_{\phi + \pi} +2R \vuu_{\phi + \pi} =  \vx_\phi\\
    &= \int_{-2R}^{- d}\left\langle\vf (\vx_{\phi} + s\vuu_{\phi + \pi}),\vuu_{\phi + \pi}^{m}\right\rangle\,ds \\
    &= (-1)^m\int_{-2R}^{- d}\left\langle\vf (\vx_{\phi} - s\vuu_{\phi}), \vuu_{\phi}^{m}\right\rangle\,ds, \quad \mbox{ since }  \vuu_{\phi + \pi} =  - \vuu_\phi\\
    &= (-1)^m\int_{d}^{2R}\left\langle\vf (\vx_{\phi} + s\vuu_{\phi}),\vuu_{\phi}^{m}\right\rangle\,ds.
\end{align*}
Hence
\begin{align}\label{Simplified expression for B1}   B_1= (-1)^m\int_{d}^{2R}\left\langle\vf (\vx_{\phi} + s\vuu_{\phi}),\vuu_{\phi}^{m}\right\rangle\,ds.
\end{align}
Now in order to complete the proof, we need to consider  $m$ is even and odd separately. 
\subsection*{For the case, when $m$ is even}
 Now after adding the integrals $A_1$ from equation \eqref{Sepration of Lf} and $B_1$ from equation \eqref{Simplified expression for B1}, and using the equation \eqref{eq:relation between VLT and regular RT}, we have 
\begin{align}\label{data st line passing through origin}
    A_1 + B_1 &= \int_{0}^{d}\left\langle\vf (\vx_\phi +s\vuu_\phi),\vuu_{\phi}^{m} \right\rangle\,ds + \int_{d}^{2R}\left\langle\vf (\vx_{\phi} + s\vuu_{\phi}),\vuu_{\phi}^{m}\right\rangle\,ds =  \int_{0}^{2R}\left\langle\vf (\vx_\phi +s\vuu_\phi),\vuu_{\phi}^{m} \right\rangle\,ds \nonumber\\
    &= \Ic \vf \left(\phi + \pi/2, 0\right) =  \Lc\vf(\phi, 2R). 
\end{align}
Using the similar arguments for   $A_2$ from equation \eqref{Sepration of Lf} and $B_2$ from equation \eqref{Separation of Lf with d= 2R-d}, we get the following identity: 
\begin{align}\label{data st line}
    A_2 + B_2 = \mathcal{I}\vf\left(\psi_\phi, t_d \right), \quad \mbox { with } \psi_{\phi} = \phi + \theta + \pi/2 \ \mbox { and } \  t_{d} = (R-d)\sin(\pi + \theta).
\end{align}
Now using the equations \eqref{Sepration of Lf}, \eqref{Separation of Lf with d= 2R-d}, \eqref{data st line passing through origin} and \eqref{data st line}, we have
\begin{align*}
\Lc\vf(\phi,d) +  \Lc\vf(\phi + \pi,2R-d) = \mathcal{I}\vf({\phi + \pi/2,0}) + \mathcal{I}\vf{{(\psi_{\phi},t_{d})}}
\end{align*}
which  implies that
\begin{align}\label{eq:LRT in terms of LVT}
\Ic\vf{{(\psi_{\phi},t_{d})}} = \Lc\vf(\phi,d) +  \Lc\vf(\phi + \pi,2R-d) - \Lc \vf(\phi, 2R).
\end{align}
 
Using the similar arguments as used for longitudinal case, we have the following  analogous relations for the transverse ray transform  $\Jc \vf$ and mixed ray transforms $\Kc^{(k)}\vf$ with that of transverse V-line transforms $\Tc \vf$ and mixed V-line transforms $\Mc^{(k)}\vf$, $k= 1,2,\cdots ,m-1$:
\begin{align}\label{eq:TRT in terms of TVT}
\Jc\vf(\psi_{\phi},t_{d}) = \Tc\vf(\phi,d) + \Tc\vf(\phi + \pi,2R-d) - \Tc\vf({\phi + \pi/2,0}),
\end{align}
and
\begin{align}\label{eq:MRT in terms of MVT}
(-1)^{k}\Kc^{(k)}\vf(\psi_{\phi},t_{d}) = \Mc^{(k)}\vf(\phi,d) + \Mc^{(k)}\vf(\phi + \pi,2R-d) - (-1)^{k}\Mc^{(k)}\vf({\phi + \pi/2,0}).
\end{align}
\subsection*{For the case, when $m$ is odd}
 By using a similar type of analysis as for the case, when $m$ is even, we have 
\begin{align}\label{odd case}
    \begin{cases}
     \Ic\vf{{(\psi_{\phi},t_{d})}} = \Lc\vf(\phi,d) -  \Lc\vf(\phi + \pi,2R-d) - \Lc \vf(\phi, 2R),\\
     \Jc\vf(\psi_{\phi},t_{d}) = -\Tc\vf(\phi,d) + \Tc\vf(\phi + \pi,2R-d) - \Tc\vf({\phi + \pi/2,0}), \\
    \Kc^{(k)}\vf(\psi_{\phi},t_{d}) = (-1)^{k}\Mc^{(k)}\vf(\phi,d) - (-1)^{k}\Mc^{(k)}\vf(\phi + \pi,2R-d) - \Mc^{(k)}\vf({\phi + \pi/2,0}). 
    \end{cases}
\end{align}
Now for each $m\in \mathbb{N}$, we notice that the right-hand side of the above relations given in equations \eqref{eq:LRT in terms of LVT}, \eqref{eq:TRT in terms of TVT}, \eqref{eq:MRT in terms of MVT} and \eqref{odd case} are completely known in terms of the provided V-line data, hence the left-hand sides of these relations which represent the ray transforms, namely the longitudinal, transverse, and mixed ray transform of $\vf$, along the straight lines defined by the parameters $(\psi_\phi, t_d)$ are also known. Using \eqref{data st line} (see \ref{fig: def of broken ray}), we can easily observe that we have no data outside of the disk $\mathbb{D}_{R\sin \theta}$. Therefore, by varying the parameter $(\psi_\phi, t_d)$, we have the information about the longitudinal, transverse and mixed ray transforms of $\vf$ along every line passing through $\mathbb{D}_{R\sin \theta}.$ From the knowledge of these longitudinal, mixed, and transverse ray transforms of $\vf$, we give an inversion formula for each components of $\vf$ by following the techniques  used in \cite[Section 5.2]{Derevtsov_2015}.

As mentioned before, we have longitudinal, transverse, and mixed ray transforms of $\vf$ are known. To begin, we rewrite this data in terms of the  Radon transforms of the components of $\vf$:
\begin{align}\label{system}
\begin{cases}
\Ic \vf(\psi_\phi, t_d) = (w^{\perp})^{i_1}\cdots(w^{\perp})^{i_{m-1}}(w^{\perp})^{i_m}\Rc f_{i_{1}i_{2}\cdots i_{m}}(\psi_\phi, t_d),\\
  \Kc^2 \vf(\psi_\phi, t_d) = (w^{\perp})^{i_1}\cdots (w^{\perp})^{i_{m-1}}w^{i_m}\Rc f_{i_{1}i_{2}\cdots i_{m}}(\psi_\phi, t_d),\\
  \cdots \hspace{.6cm}\cdots \hspace{.6cm}\cdots\\
  \Kc^{(m-1)} \vf(\psi_\phi, t_d) = (w^{\perp})^{i_1}w^{i_2}\cdots w^{i_{m-1}}w^{i_m}\Rc f_{i_{1}i_{2}\cdots i_{m}}(\psi_\phi, t_d),\\
  \Jc\vf(\psi_\phi, t_d) = w^{i_1}\cdots w^{i_{m-1}}w^{i_m}\Rc f_{i_{1}i_{2}\cdots i_{m}}(\psi_\phi, t_d).
  \end{cases}
\end{align}
We can write this system of equations \eqref{system} in the matrix form $Y = A X$, where the matrix $A$ is
{\footnotesize\begin{align*}
   \begin{bmatrix}
      \left((w^{\perp})^{1}\cdots(w^{\perp})^{1}(w^{\perp})^{1}\right) & m\left((w^{\perp})^{1}\cdots(w^{\perp})^{1}(w^{\perp})^{2}\right)& \cdots &\left((w^{\perp})^{2}\cdots(w^{\perp})^{2}(w^{\perp})^{2}\right)\\
      \left((w^{\perp})^{1}\cdots(w^{\perp})^{1}(w)^{1}\right) & \left((w^{\perp})^{1}\cdots(w^{\perp})^{1}(w)^{2} + (m-1)(w^{\perp})^{1}\cdots(w^{\perp})^{2}(w)^{1}\right)& \cdots &\left((w^{\perp})^{2}\cdots(w^{\perp})^{2}(w)^{2}\right)\\
      \cdots & \cdots & \cdots & \cdots \\
      \left((w^{\perp})^{1}\cdots(w)^{1}(w)^{1}\right) & \left((w^{\perp})^{2}\cdots(w)^{1}(w)^{1}+(m-1)(w^{\perp})^{1}\cdots(w)^{1}(w)^{2} \right)& \cdots &\left((w^{\perp})^{2}\cdots(w)^{2}(w)^{2}\right)\\
      \left((w)^{1}\cdots(w)^{1}(w)^{1}\right) & m\left((w)^{1}\cdots(w)^{1}(w)^{2}\right)& \cdots &\left((w)^{2}\cdots(w)^{2}(w)^{2}\right)
  \end{bmatrix},
\end{align*}}
$A$ is the $(m+1)\times (m+1)$ order matrix and
\begin{align*}
  Y :=  \begin{bmatrix}
        \Ic \vf(\psi_\phi, t_d) \\
        \Kc^2 \vf(\psi_\phi, t_d)\\
        \cdots\\
        \Kc^{(m-1)} \vf(\psi_\phi, t_d)\\
        \Jc\vf(\psi_\phi, t_d)
    \end{bmatrix}_{(m+1)\times 1} 
    \quad 
    X :=   \begin{bmatrix}
        \Rc f_{1\cdots 11}(\psi_\phi, t_d)\\
        \Rc f_{1\cdots 12}(\psi_\phi, t_d)\\
        \cdots \\
        \Rc f_{12\cdots 22}(\psi_\phi, t_d)\\
        \Rc f_{22\cdots 22}(\psi_\phi, t_d)
          \end{bmatrix}_{(m+1)\times 1}.
\end{align*}
Now from \cite{Derevtsov_2015}, we have that $A^2 = I_{(m+1)\times (m+1)}$, where $I_{(m+1)\times (m+1)}$ is the identity matrix of order $(m+1)\times (m+1)$, hence  the matrix $A$ is invertible and   the system \eqref{system} has a unique solution for
$X$.
Thus, we have that the  Radon transform of each component of the $m$-tensor field, 
   is known and from the inversion formula for the Radon transform for functions we have the following explicit formula for each components of $m$-tensor field:
   \begin{align}\label{eq:Radon's inversion formula}
     f_{i_{1}i_{2}\cdots i_{m}}(\vx) = \frac{1}{2\pi} \left(-\Delta\right)^{1/2} \int_0^{2\pi}\Rc f_{i_{1}i_{2}\cdots i_{m}}((\cos \alpha, \sin \alpha), x \cos \alpha + y \sin \alpha) d\alpha, 
\end{align}
where the operator $(-\Delta)^{1/2}$ is defined with the help of the following property of Fourier transform, ${(-\Delta f)}^{\wedge}(\xi) = |\xi|^{2}\widehat{f}(\xi).$ Therefore, by using the inversion formula of the Radon transform \eqref{eq:Radon's inversion formula}, we obtain the inversion formula for  the symmetric $m$-tensor field $\vf$.  This completes the proof of the theorem. 
\end{proof}
 
\section{Proof of Theorem \ref{th:partial data recovery}}\label{sec:partial data case}
In this section, we give   the series formula that allows us to obtain the symmetric $2$-tensor field from the mixed, longitudinal, and transverse V-line transforms. We start with a definition and some properties of the Mellin transform, required to give the series formula for the components of the tensor fields which is the aim of   Theorem \ref{th:partial data recovery}.  
\begin{definition}[\cite{Flajolet_1995}]  The \textbf{Mellin transform} for a complex-valued integrable function $f$ that decays at infinity, denoted by $\mathcal{P}f$, is the function of complex variable $s$ and given by 
    \begin{equation*}
      \mathcal{P}f(s) := \int_{0}^{\infty} p^{s-1}f(p) \,dp.
  \end{equation*}
\end{definition}  
  The  Mellin transform satisfies the following identities:
\begin{enumerate}
\item $\displaystyle \mathcal{P}\left[r^k f(r)\right](s) = \mathcal{P}f(s+k)\label{p1}$ 
\item $\displaystyle \mathcal{P}\left[\int_{t}^{\infty}f(x)\,dx\right](s) = \frac{\mathcal{P}f(s+1)}{s}\label{p2} $ 
\end{enumerate}
\begin{figure}[H]
\centering
\begin{subfigure}[c]{0.46\textwidth}
\centering
\includegraphics[width=\textwidth]{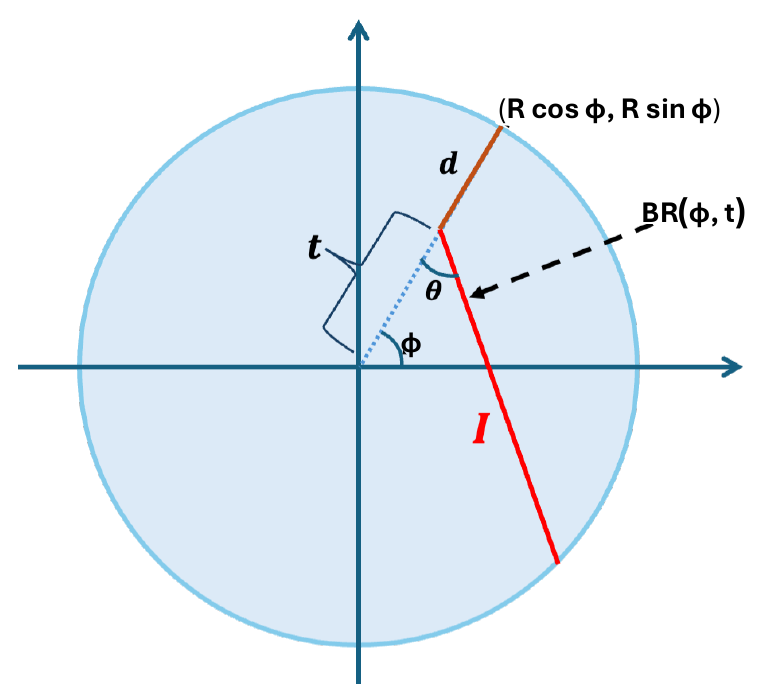}\caption{A sketch of modified parametrization of the broken line $B(\phi, t)$ with $ t = R- d$.}\label{fig: B(beta, t)}
\end{subfigure}
\hfill
\begin{subfigure}[c]{0.49\textwidth}
\centering
\includegraphics[width=\textwidth]{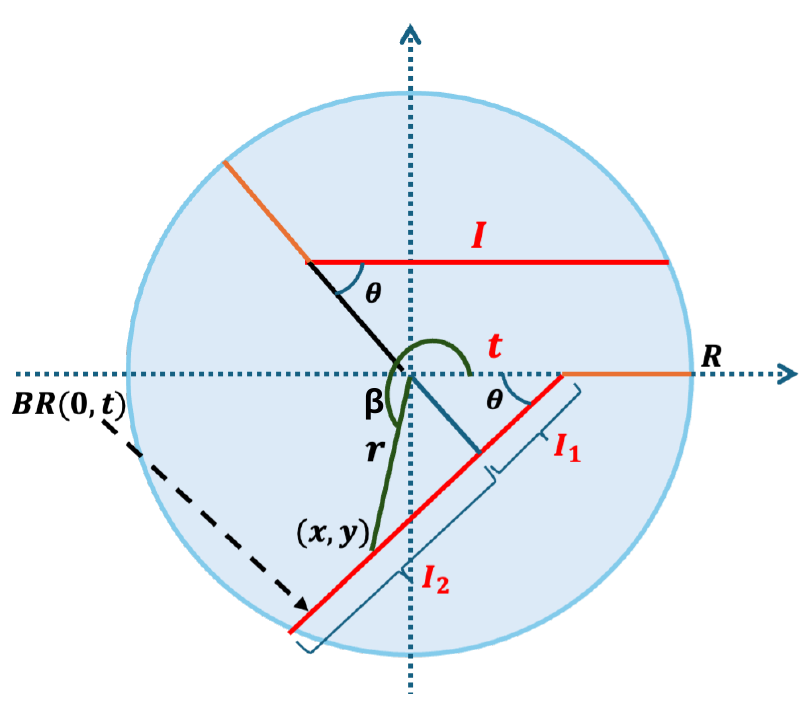}\caption{Here $B(0, t)$ is the broken line obtained by rotating $B(\phi, t)$ clockwise by an angle $\phi$.}\label{fig: polar coordinates}
\end{subfigure}
\end{figure}

In order to simplify the upcoming calculations, we reparametrize the broken lines  described by two parameters: $\phi\in [0,2\pi)$ and $d\in [0,R]$ to 
 $\phi\in [0,2\pi)$ and the  new parameter $t = R - d$  where $d$ denotes the distance travelled along the diameter before scattering.  In other words, broken rays are now parameterized by the ordered pair $(\phi, t)$ as shown in \ref{fig: B(beta, t)}.

Let us denote  
\begin{align*}
    \vf(\beta,r) :=  \begin{bmatrix}
        f_{11}(\beta,r) & f_{12}(\beta,r)\\
        f_{12}(\beta,r) & f_{22}(\beta,r)
    \end{bmatrix}   
\end{align*} the unknown symmetric $2$-tensor field $\vf$  in polar coordinates and  \[\Fc(\phi, t):= \mathcal{L}\vf(\phi, t),\ \Gc(\phi, t):= \mathcal{T}\vf(\phi, t), \ \mbox{and}\ \Hc(\phi, t):= \mathcal{M}\vf(\phi, t).  \] Then the Fourier series of $f_{11}(\beta,r)$, $f_{12}(\beta,r)$, $f_{22}(\beta,r)$, $\Fc(\phi, t)$,  $\Gc(\phi, t)$ and $\Hc(\phi, t)$ are given by:
\begin{align}
     f_{11}(\beta,r) = \sum_{n=-\infty}^{\infty} a_n(r)e^{i n\beta}, \quad \mbox{ with } a_n(r) = \frac{1}{2\pi}\int_{0}^{2\pi} f_{11}(\beta,r)e^{-in\beta}\,d\beta\label{eq:Fourier series of f11} \\
     f_{12}(\beta,r) = \sum_{n=-\infty}^{\infty} b_n(r)e^{i n\beta}, \ \quad \mbox{ with }  b_n(r) = \frac{1}{2\pi}\int_{0}^{2\pi} f_{12}(\beta,r)e^{-in\beta}\,d\beta\label{eq:Fourier series of f12} \\
     f_{22}(\beta,r) = \sum_{n=-\infty}^{\infty} c_n(r)e^{i n\beta}, \ \quad \mbox{ with }  c_n(r) = \frac{1}{2\pi}\int_{0}^{2\pi} f_{22}(\beta,r)e^{-in\beta}\,d\beta\label{eq:Fourier series of f22} \\
    \Fc(\phi, t) = \sum_{n=-\infty}^{\infty} F_n(t) e^{in\phi}, \ \ \quad \mbox{ with } F_n(t) = \frac{1}{2\pi}\int_{0}^{2\pi} \Fc(\phi, t)e^{-in\phi}\,d\phi \label{eq:Fourier series of LVT}\\
    \Gc(\phi, t) = \sum_{n=-\infty}^{\infty} G_n(t) e^{in\phi}, \ \ \quad \mbox{ with } G_n(t) = \frac{1}{2\pi}\int_{0}^{2\pi} \Gc(\phi, t)e^{-in\phi}\,d\phi \label{eq:Fourier series of TVT} \\
     \Hc (\phi, t) = \sum_{n=-\infty}^{\infty} H_n(t) e^{in\phi}, \ \ \quad \mbox{ with } H_n(t) = \frac{1}{2\pi}\int_{0}^{2\pi} \Hc(\phi, t)e^{-in\phi}\,d\phi \label{eq:Fourier series of MVT}.
\end{align}
 
Firstly, we show that the Fourier coefficients of each component of the $2$-tensor field can be expressed in terms of the Fourier coefficients of the longitudinal, transverse and mixed V-line transforms and using the inversion formula for Mellin-transform, we conclude the proof of  main Theorem \ref{th: Melin transforms for Fourier coefficients}. In the next proposition, we prove the relation between the Fourier coefficients of $2$-tensor fields in terms of the Fourier coefficients of the given data. 
\begin{prop} If $F_n$, $G_n$, and $H_n$ are the Fourier coefficients of the Fourier series as determined above for the longitudinal, transverse and mixed V-line respectively, and $a_n$, $b_n$, and $c_n$ are the $n$th Fourier coefficients of the  components $f_{11},f_{12} $ and $f_{22}$ of the $2$-tensor field $\vf$ respectively. We have the following relations holds: 
    \begin{align}\label{eq:relation between a_n, F_n, G_n and H_n}
   \left( 2 F_n(t) + F_{n+2}(t) +F_{n-2}(t)\right) +& \left( 2 G_n(t) - G_{n+2}(t) - G_{n-2}(t)\right) +2i \left(  H_{n-2}(t)  - H_{n+2}(t)\right) \nonumber\\ &= 4 i\left( \int_{t}^{R} a_{n}(r) e^{in\beta}\,ds  + \int_I a_{n}(r) e^{2i\theta} e^{in\beta}\,ds\right),
\end{align}

\begin{align}\label{eq:relation between b_n, F_n, G_n and H_n}
   \left(  F_{n-2}(t) - F_{n+2}(t)\right) - & \left(  G_{n-2}(t) - G_{n+2}(t)\right) +2i \left(  H_{n-2}(t)  + H_{n+2}(t)\right) \nonumber\\ &= 4 i \left(\int_{t}^{R} b_{n}(r) e^{in\beta}\,ds  + \int_I b_{n}(r) e^{2i\theta} e^{in\beta}\,ds\right)
\end{align}
and
\begin{align}\label{eq:relation between c_n, F_n, G_n and H_n}
   \left( 2 F_{n}(t)- F_{n-2}(t) - F_{n+2}(t)\right) +& \left(  2 G_n(t)+ G_{n-2}(t) + G_{n+2}(t)\right) +2i \left(  H_{n+2}(t)  - H_{n-2}(t)\right) \nonumber\\ &= 4  \left(\int_{t}^{R} c_{n}(r) e^{in\beta}\,ds  + \int_I c_{n}(r) e^{2i\theta} e^{in\beta}\,ds\right).
\end{align}
\end{prop}
\begin{proof}
Let $\widetilde{\Theta}(\phi, t)$ denote the unit vector field along the broken ray $BR(\phi,t)$. More precisely,   $\widetilde\Theta(\phi, t) = \vuu_\phi = -(\cos \phi, \sin \phi)$ for the first section of the broken ray that starts from the boundary of the disk $\partial\mathbb{D}_R$ and $\widetilde{\Theta}(\phi, t) = \vv_\phi = -(\cos (\theta + \phi), \sin (\theta +\phi))$ along the second section of the broken ray.
We have 
\begin{align*}
\Fc(\phi,t) &=\Lc \vf(\phi,t)= \int_{BR(\phi,t)} \left\langle \vf(\beta,r),\widetilde{\Theta}_{\phi}^{2}\right\rangle\,ds
= \int_{BR(0,t)} \left\langle\vf\left(\beta+\phi,r \right), \widetilde{\Theta}_{\phi}^{2}\right\rangle\,ds\\
&= \int_t^R  \left\langle \vf(\phi,r),{\vuu}_{\phi}^{2}\right\rangle\,dr  + \int_I \left\langle \vf(\beta + \phi ,r),{\vv}_{\phi}^{2}\right\rangle\,ds.
\end{align*}
In the above expression, the first integral is along the $x$-axis (hence the angle $\beta= 0$ and $dr$ be the length measure along the corresponding line), and the second integral is along the other section of the broken ray $B(0, t)$. 
\begin{align*}
 &\Fc(\phi,t)  = \sum_{n=-\infty}^{\infty}  \int_{t}^{R} \left(a_n(r)u_{1}^{2} + 2b_n(r)u_{1}u_{2} + c_{n}(r)u_{2}^{2}\right)e^{in \phi}\,dr \\ 
& + \sum_{n=-\infty}^{\infty}\int_{I} \left(a_n(r)v_{1}^{2} + 2b_n(r)v_{1}v_{2} + c_{n}(r)v_{2}^{2}\right) {e}^{in(\beta + \phi)}\,ds\\
&= \sum_{n=-\infty}^{\infty}  \int_{t}^{R} \left(a_n(r)\cos^{2}{\phi} + 2b_n(r)\cos{\phi}\sin{\phi} + c_{n}(r)\sin^{2}{\phi}\right)e^{in \phi}\,dr \\ 
& + \sum_{n=-\infty}^{\infty}\int_{I} \left(a_n(r)\cos^{2}{(\phi + \theta)} + 2b_n(r)\cos{(\phi + \theta)}\sin{(\phi + \theta)} + c_{n}(r)\sin^{2}{(\phi + \theta)}\right) {e}^{in(\beta + \phi)}\,ds.
\end{align*}
Now, putting the value of $\displaystyle \cos \theta = \frac{e^{i\theta} + e^{-i\theta}}{2}$ and  $\displaystyle \sin \theta = \frac{e^{i\theta} - e^{-i\theta}}{2i} $ in the above expression of $\Fc(\phi, t)$, and after simplification, we have
\begin{align*}
4\Fc(\phi, t) &=\sum_{n=-\infty}^{\infty}  \int_{t}^{R} \left\{a_n(r) - c_n(r) - 2i b_n(r)\right\}e^{in\beta}e^{i(n+2) \phi}\,dr \\ 
& \quad + \sum_{n=-\infty}^{\infty}  \int_{t}^{R} \left\{a_n(r) - c_n(r) +2i b_n(r)\right\}e^{in\beta}e^{i(n-2) \phi}\,dr \\ 
& \quad + 2\sum_{n=-\infty}^{\infty}  \int_{t}^{R} \left\{a_n(r) + c_n(r) \right\}e^{in\beta}e^{in \phi}\,dr \\ 
& \quad + \sum_{n=-\infty}^{\infty}  \int_I \left\{a_n(r) - c_n(r) - 2i b_n(r)\right\}e^{2i\theta}e^{in\beta}e^{i(n+2) \phi}\,dr \\ 
& \quad + \sum_{n=-\infty}^{\infty}  \int_I \left\{a_n(r) - c_n(r) +2i b_n(r)\right\}e^{2i\theta}e^{in\beta}e^{i(n-2) \phi}\,dr \\ 
& \quad + 2\sum_{n=-\infty}^{\infty}  \int_I \left\{a_n(r) + c_n(r) \right\}e^{2i\theta}e^{in\beta}e^{in \phi}\,dr\\
&=\sum_{n=-\infty}^{\infty}  \int_{t}^{R} \left\{a_{n-2}(r) - c_{n-2}(r) - 2i b_{n-2}(r)\right\}e^{i(n-2)\beta}e^{in \phi}\,dr \\ 
& \quad + \sum_{n=-\infty}^{\infty}  \int_{t}^{R} \left\{a_{n+2}(r) - c_{n+2}(r) +2i b_{n+2}(r)\right\}e^{i(n+2)\beta}e^{in \phi}\,dr \\ 
& \quad + 2\sum_{n=-\infty}^{\infty}  \int_{t}^{R} \left\{a_n(r) + c_n(r) \right\}e^{in\beta}e^{in \phi}\,dr \\ 
& \quad + \sum_{n=-\infty}^{\infty}  \int_I \left\{a_{n-2}(r) - c_{n-2}(r) - 2i b_{n-2}(r)\right\}e^{2i\theta}e^{i(n-2)\beta}e^{in \phi}\,dr \\ 
& \quad + \sum_{n=-\infty}^{\infty}  \int_I \left\{a_{n+2}(r) - c_{n+2}(r) +2i b_{n+2}(r)\right\}e^{2i\theta}e^{i(n+2)\beta}e^{in \phi}\,dr \\ 
& \quad + 2\sum_{n=-\infty}^{\infty}  \int_I \left\{a_n(r) + c_n(r) \right\}e^{2i\theta}e^{in\beta}e^{in \phi}\,dr.
\end{align*}
By comparing the previous relation with the Fourier series $\displaystyle \Fc(\phi, t)= \sum_{n=-\infty}^{\infty} F_n(t) e^{in\phi}$, we obtain the following expression  
 \begin{align}\label{eq:F_n}
4 F_n(t)   &=  \int_{t}^{R} \left\{a_{n-2}(r) - c_{n-2}(r) - 2i b_{n-2}(r)\right\}e^{i(n-2)\beta}\,dr \nonumber\\ 
& \quad +   \int_{t}^{R} \left\{a_{n+2}(r) - c_{n+2}(r) +2i b_{n+2}(r)\right\}e^{i(n+2)\beta}\,dr \nonumber \\ 
& \quad + 2  \int_{t}^{R} \left\{a_n(r) + c_n(r) \right\}e^{in\beta}\,dr \nonumber\\ 
& \quad +  \int_I \left\{a_{n-2}(r) - c_{n-2}(r) - 2i b_{n-2}(r)\right\}e^{2i\theta}e^{i(n-2)\beta}\,dr \nonumber\\ 
& \quad +   \int_I \left\{a_{n+2}(r) - c_{n+2}(r) +2i b_{n+2}(r)\right\}e^{2i\theta}e^{i(n+2)\beta}\,dr \nonumber\\ 
& \quad + 2  \int_I \left\{a_n(r) + c_n(r) \right\}e^{2i\theta}e^{in\beta}\,dr.
 \end{align}
By repeating similar type of calculations and comparing the coefficients of the Fourier series of  $\displaystyle \Gc(\phi, t)= \sum_{n=-\infty}^{\infty} G_n(t) e^{in\phi}$ and $\displaystyle \Hc(\phi, t)= \sum_{n=-\infty}^{\infty} H_n(t) e^{in\phi}$, we have the following identities:
 \begin{align}\label{eq:G_n}
4 G_n(t)   &=  -\int_{t}^{R} \left\{a_{n-2}(r) - c_{n-2}(r) - 2i b_{n-2}(r)\right\}e^{i(n-2)\beta}\,dr \nonumber\\ 
& \quad +   \int_{t}^{R} \left\{-a_{n+2}(r) + c_{n+2}(r) -2i b_{n+2}(r)\right\}e^{i(n+2)\beta}\,dr \nonumber \\ 
& \quad + 2  \int_{t}^{R} \left\{a_n(r) + c_n(r) \right\}e^{in\beta}\,dr \nonumber\\ 
& \quad - \int_I \left\{a_{n-2}(r) - c_{n-2}(r) - 2i b_{n-2}(r)\right\}e^{2i\theta}e^{i(n-2)\beta}\,dr \nonumber\\ 
& \quad +   \int_I \left\{-a_{n+2}(r) + c_{n+2}(r) -2i b_{n+2}(r)\right\}e^{2i\theta}e^{i(n+2)\beta}\,dr \nonumber\\ 
& \quad + 2  \int_I \left\{a_n(r) + c_n(r) \right\}e^{2i\theta}e^{in\beta}\,dr
 \end{align}
 and 
 \begin{align}\label{eq:H_n}
4 H_n(t)  &=  \int_{t}^{R} \left\{ 2 b_{(n-2)}(r) + i \left( a_{(n-2)}(r) - c_{(n-2)}(r)\right)\right\}e^{i(n-2)\beta}\,dr \nonumber \\ 
& \quad +   \int_{t}^{R} \left\{ 2 b_{(n+2)}(r) - i \left( a_{(n+2)}(r) - c_{(n+2)}(r)\right)\right\}e^{i(n+2)\beta}\,dr \nonumber\\ 
& \quad + \int_I \left\{ 2 b_{(n-2)}(r) + i \left( a_{(n-2)}(r) - c_{(n-2)}(r)\right)\right\}e^{2i\theta}e^{i(n-2)\beta}\,dr \nonumber\\ 
& \quad + \int_I \left\{ 2 b_{(n+2)}(r) - i \left( a_{(n+2)}(r) - c_{(n+2)}(r)\right)\right\}e^{2i\theta}e^{i(n+2)\beta}\,dr .
 \end{align}
Adding equations \eqref{eq:F_n} and \eqref{eq:G_n}, we have
 \begin{align}\label{eq:F_n + G_n}
       F_n(t) + G_n(t)  = \int_{t}^{R} \left[a_{n}(r) + c_n(r) \right]e^{in\beta}\,dr + \int_{I} \left[a_{n}(r) + c_n(r) \right] e^{in\beta}e^{2i\theta}\,ds,
 \end{align}
 and after subtracting equation \eqref{eq:G_n} from equation \eqref{eq:F_n}, we get
 \begin{align}\label{eq:F_n - G_n}
       2\left(F_n(t) - G_n(t)\right)  &=  \int_{t}^{R} \left\{a_{n-2}(r) - c_{n-2}(r) - 2i b_{n-2}(r)\right\}e^{i(n-2)\beta}\,dr \nonumber\\ 
& \quad +   \int_{t}^{R} \left\{a_{n+2}(r) - c_{n+2}(r) +2i b_{n+2}(r)\right\}e^{i(n+2)\beta}\,dr \nonumber \\ 
& \quad +  \int_I \left\{a_{n-2}(r) - c_{n-2}(r) - 2i b_{n-2}(r)\right\}e^{2i\theta}e^{i(n-2)\beta}\,ds \nonumber\\ 
& \quad +   \int_I \left\{a_{n+2}(r) - c_{n+2}(r) +2i b_{n+2}(r)\right\}e^{2i\theta}e^{i(n+2)\beta}\,ds.
 \end{align}
 Now multiply equation \eqref{eq:H_n} with $i$ and add it to equation \eqref{eq:F_n - G_n}, we have
 \begin{align}\label{F_n - G_n + 2iH_n}
    F_n(t) - G_n(t) + 2 i  H_n(t) &=  \int_{t}^{R} \left\{a_{n+2}(r) - c_{n+2}(r) +2i b_{n+2}(r)\right\}e^{i(n+2)\beta}\,dr \nonumber \\ 
 & \quad +   \int_I \left\{a_{n+2}(r) - c_{n+2}(r) +2i b_{n+2}(r)\right\}e^{2i\theta}e^{i(n+2)\beta}\,ds.
\end{align}
Multiply equation \eqref{eq:H_n} with $i$ and subtract it from equation \eqref{eq:F_n - G_n}, we have
\begin{align}\label{F_n - G_n - 2iH_n}
    F_n(t) - G_n(t) - 2 i  H_n(t) &= \int_{t}^{R} \left\{a_{n-2}(r) - c_{n-2}(r) - 2i b_{n-2}(r)\right\}e^{i(n-2)\beta}\,dr \nonumber\\ 
& \quad + \int_I \left\{a_{n-2}(r) - c_{n-2}(r) - 2i b_{n-2}(r)\right\}e^{2i\theta}e^{i(n-2)\beta}\,ds
\end{align}
Equations \eqref{F_n - G_n + 2iH_n} and \eqref{F_n - G_n - 2iH_n} can be further rewritten in the following form:
\begin{align}\label{F_{n-2} - G_n + 2iH_n}
    F_{n-2}(t) - G_{n-2}(t) + 2 i  H_{n-2}(t) &=  \int_{t}^{R} \left\{a_{n}(r) - c_{n}(r) +2i b_{n}(r)\right\}e^{in\beta}\,dr \nonumber \\ 
 & \quad +   \int_I \left\{a_{n}(r) - c_{n}(r) +2i b_{n}(r)\right\}e^{2i\theta}e^{in\beta}\,ds.
\end{align}
and
\begin{align}\label{F_{n+2} - G_n - 2iH_n}
    F_{n+2}(t) - G_{n+2}(t) - 2 i  H_{n+2}(t) &=  \int_{t}^{R} \left\{a_{n}(r) - c_{n}(r) -2i b_{n}(r)\right\}e^{in\beta}\,dr \nonumber \\ 
 & \quad +   \int_I \left\{a_{n}(r) - c_{n}(r) - 2i b_{n}(r)\right\}e^{2i\theta}e^{in\beta}\,ds.
\end{align}
By simplifying the above two expressions in  \eqref{F_{n+2} - G_n - 2iH_n}, \eqref{F_{n-2} - G_n + 2iH_n} and using \eqref{eq:F_n + G_n}, we get
\begin{align*}
   \left( 2 F_n(t) + F_{n+2}(t) +F_{n-2}(t)\right) +& \left( 2 G_n(t) - G_{n+2}(t) - G_{n-2}(t)\right) +2i \left(  H_{n-2}(t)  - H_{n+2}(t)\right) \nonumber\\ &= 4 i\left( \int_{t}^{R} a_{n}(r) e^{in\beta}\,ds  + \int_I a_{n}(r) e^{2i\theta} e^{in\beta}\,ds\right)
\end{align*}
and
\begin{align*}
   \left(  F_{n-2}(t) - F_{n+2}(t)\right) - & \left(  G_{n-2}(t) - G_{n+2}(t)\right) +2i \left(  H_{n-2}(t)  + H_{n+2}(t)\right) \nonumber\\ &= 4 i \left(\int_{t}^{R} b_{n}(r) e^{in\beta}\,ds  + \int_I b_{n}(r) e^{2i\theta} e^{in\beta}\,ds\right)
\end{align*}
and
\begin{align*}
   \left( 2 F_{n}(t)- F_{n-2}(t) - F_{n+2}(t)\right) +& \left(  2 G_n(t)+ G_{n-2}(t) + G_{n+2}(t)\right) +2i \left(  H_{n+2}(t)  - H_{n-2}(t)\right) \nonumber\\ &= 4  \left(\int_{t}^{R} c_{n}(r) e^{in\beta}\,ds  + \int_I c_{n}(r) e^{2i\theta} e^{in\beta}\,ds\right).
\end{align*}
\end{proof}
As mentioned earlier, our main goal is to reconstruct symmetric $2$-tensor fields $\vf$ from the combination of $\Lc \vf$, $\Tc \vf$ and $\Mc \vf$. For this, we first show that the Mellin transform $a_n$, $b_n$ and $c_n$ can be explicitly expressed in terms of the Mellin transform of $F_n$, $G_n$ and $H_n$. In the end, by using the inversion formulas for the Mellin transform, we recover $a_n$, $b_n$ and $c_n$ explicitly. 
\begin{theorem}\label{th: Melin transforms for Fourier coefficients}
For $n \in \mathbb{Z}$, let $a_n$, $b_n$ and $c_n$ are the Fourier coefficients (defined in \eqref{eq:Fourier series of f11}, \eqref{eq:Fourier series of f12}and \eqref{eq:Fourier series of f22}) of components of a symmetric $2$-tensor field $\displaystyle \vf  \in C_0^\infty\left(\DR, S^2(\mathbb{R}^2)\right)$. Then, we have
\begin{align*}
\mathcal{P}a_n(s) = \frac{\Pc \left(2F_{n}+ F_{n+2} + F_{n-2} + 2 G_n - G_{n+2} \right)(s-1) +2i \Pc \left(  2 H_{n-2}  - H_{n+2}\right)(s-1)  }{4\left[\frac{1}{s-1} + \mathcal{P}h_n(s-1)\right]}, 
\end{align*} 
\begin{align*}
\mathcal{P}b_n(s) = \frac{2\Pc \left(   H_{n-2} + H_{n+2}\right)(s-1) - i \Pc \left( F_{n-2} - F_{n+2}\right) +\left(  2 G_n - G_{n+2} \right)(s-1)  }{4\left[\frac{1}{s-1} + \mathcal{P}h_n(s-1)\right]}, 
\end{align*}  
and
\begin{align*}
\mathcal{P}c_n(s) = \frac{\Pc \left(2F_{n}- F_{n-2} - F_{n+2} + 2 G_n + G_{n-2} + G_{n+2}\right)(s-1) +2i \Pc \left( H_{n+2}  - H_{n-2}\right)(s-1)  }{4\left[\frac{1}{s-1} + \mathcal{P}h_n(s-1)\right]}, 
\end{align*} 
for $Re(s) > 1$,
where the piece-wise function $h_n$ is defined as follows: 
\begin{equation*}
h_{n}(t) = \begin{cases}
(-1)^n e^{i\theta} e^{in{\psi(t)}}\frac{1+t\cos[\psi(t)] + t^2 \sin{[\psi(t)]} \frac{\sin{\theta}}{\sqrt{1-t^2 \sin^2(\theta)}}}{\sqrt{1+t^2 + 2t\cos{[\psi(t)]}}}, & \quad \mbox{ for } 0< t \leq 1\\
(-1)^n e^{i\theta}e^{in{\psi(t)}}\frac{1+t\cos[\psi(t)] + t^2 \sin{[\psi(t)]} \frac{\sin{\theta}}{\sqrt{1-t^2 \sin^2(\theta)}}}{\sqrt{1+t^2 + 2t\cos{[\psi(t)]}}}\\ 
\quad - e^{i\theta} e^{in[2\theta - {\psi(t)}]}\frac{1-t\cos[2\theta - \psi(t)] + t^2 \sin{[2\theta - \psi(t)]} \frac{\sin{\theta}}{\sqrt{1-t^2 \sin^2(\theta)}}}{\sqrt{1+t^2 - 2t\cos{[2\theta - \psi(t)]}}},   &  \quad \mbox{ for } 1<t < \frac{1}{\sin{\theta}}  \\
0, & \quad \mbox{ for } t > \frac{1}{\sin{\theta}}
\end{cases}
\end{equation*}
with $\psi(t) = \sin^{-1}{(t\sin{\theta})} + \theta$. 
\end{theorem}

\begin{proof} 


First, we split the line segment $I$ in two sections $I_1$ and $I_2$, where $ I_1 := \{(x,y) : y = (x-t)\tan{\theta},~ t\sin^2{\theta} \leq x \leq t \}$ and  $ I_2 := \{(x,y) : y = (x-t)\tan{\theta},~ -\infty \leq x \leq t\sin^2{\theta} \}$ (see \ref{fig: B(beta, t)}). The following representations, for the polar angles $ \beta(r)$ that are obtained by a point $(x,y)\in I$:
\begin{align*}
    \beta(r) = \left\{\begin{array}{cc}
     2 \theta -  \psi(t/r)  ,     & \quad \mbox{ for }(x, y) \in I_1 \\
     \pi +   \psi(t/r) ,   & \quad \mbox{ for }(x, y) \in I_2
    \end{array},\right. 
\end{align*}
where $\psi(t) = \sin^{-1}\left(t \sin \theta\right) +  \theta$, and 
\begin{align*}
    ds = \frac{r - t\cos{\beta} + tr \frac{d\beta}{dr}\sin{\beta}}{\sqrt{{r}^2 + t^2 - 2{r} t \cos{\beta}}}\,dr.
\end{align*}
where $ds$ be the length measures for the line interval $I$ (please refer to \cite[Theorem 2]{Ambartsoumian_2013} for more details).

Hence we have
\begin{align*}
ds = \begin{cases}
\frac{r-t\cos[2\theta - \psi(\frac{t}{r})] + \frac{t^2}{r} \sin{[2\theta - \psi(\frac{t}{r})]} \frac{\sin{\theta}}{\sqrt{1-\frac{t^2}{r^2} \sin^2\theta}}}{\sqrt{r^2+t^2 - 2tr\cos{[2\theta - \psi(\frac{t}{r})]}}}dr,  & \quad \mbox{ for } (x,y) \in I_1\\
\frac{r - t\cos[\pi + \psi(\frac{t}{r})] - \frac{t^2}{r} \sin{[\pi + \psi(\frac{t}{r})]} \frac{\sin{\theta}}{\sqrt{1-\frac{t^2}{r^2} \sin^2\theta}}}{\sqrt{r^2+t^2 - 2tr\cos{[\pi + \psi(\frac{t}{r})]}}}dr, & \quad \mbox{ for } (x,y) \in I_2.
\end{cases}
\end{align*}
After putting the value of $\beta$ and $ds$ into the expression \eqref{eq:relation between a_n, F_n, G_n and H_n}, we get
\begin{align*}
&\frac{1}{4i} \left[\left( 2 F_n(t) + F_{n+2}(t) +F_{n-2}(t)\right) + \left( 2 G_n(t) - G_{n+2}(t) - G_{n-2}(t)\right) +2i \left(  H_{n-2}(t)  - H_{n+2}(t)\right)\right]\\&= \int_{t}^{R} a_n(r)\,dr + \int_{t}^{t\sin{\theta}}a_n(r) e^{i\theta} e^{in[2\theta - {\psi(\frac{t}{r})}]}\frac{r-t\cos[2\theta - \psi(\frac{t}{r})] + \frac{t^2}{r} \sin{[2\theta - \psi(\frac{t}{r})]} \frac{\sin{\theta}}{\sqrt{1-\frac{t^2}{r^2} \sin^2\theta}}}{\sqrt{r^2+t^2 - 2tr\cos{[2\theta - \psi(\frac{t}{r})]}}}\,dr\\ & \quad + \int_{t\sin{\theta}}^{\infty}a_n(r) e^{i\theta} e^{in[\pi + {\psi(\frac{t}{r})}]}\frac{r - t\cos[\pi + \psi(\frac{t}{r})] - \frac{t^2}{r} \sin{[\pi + \psi(\frac{t}{r})]} \frac{\sin{\theta}}{\sqrt{1-\frac{t^2}{r^2} \sin^2\theta}}}{\sqrt{r^2+t^2 - 2tr\cos{[\pi + \psi(\frac{t}{r})]}}}\,dr \\     
 &= \int_{t}^{R} a_n(r)\,dr - \int_{t\sin{\theta}}^{t}a_n(r) e^{i\theta} e^{in[2\theta - {\psi(\frac{t}{r})}]}\frac{1-\frac{t}{r}\cos[2\theta - \psi(\frac{t}{r})] + \frac{t^2}{r^2} \sin{[2\theta - \psi(\frac{t}{r})]} \frac{\sin{\theta}}{\sqrt{1-\frac{t^2}{r^2} \sin^2\theta}}}{\sqrt{1+\frac{t^2}{r^2} - 2\frac{t}{r}\cos{[2\theta - \psi(\frac{t}{r})]}}}\,dr\\ & \quad + (-1)^n\int_{t\sin{\theta}}^{\infty}a_n(r) e^{i\theta} e^{in[{\psi(\frac{t}{r})}]}\frac{1 + \frac{t}{r}\cos[\psi(\frac{t}{r})] + \frac{t^2}{r^2} \sin{[\psi(\frac{t}{r})]} \frac{\sin{\theta}}{\sqrt{1-\frac{t^2}{r^2} \sin^2\theta}}}{\sqrt{1+\frac{t^2}{r^2} + 2\frac{t}{r}\cos{[\psi(\frac{t}{r})]}}}\,dr \\ 
&= \int_{t}^{R} a_n(r)\,dr + \{[r a_n(r)]\times h_n\}(t) 
\end{align*}
where
\begin{equation*}
\{f\times g\}(s) = \int_{0}^{\infty} f(r)g\left(\frac{s}{r}\right)\frac{dr}{r}.
\end{equation*}
Now, applying the Mellin transform on both sides in the last expression, and using the  properties \eqref{p1} and \eqref{p2} of the Mellin transform, we get 
\begin{align*}
\Pc \left(2F_{n}+ F_{n+2} + F_{n-2} + 2 G_n - G_{n+2} \right)&(s-1) + 2i \Pc \left(  2 H_{n-2}  - H_{n+2}\right)(s)\\  &= 4  \left[\frac{1}{s}\Pc a_n(s+1) + \Pc a_n(s+1)\Pc h_n(s)\right].
\end{align*}
Finally, we have the following expression for $a_n$ in terms of Mellin transform of given data
\begin{align*}
\mathcal{P}a_n(s) = \frac{\Pc \left(2F_{n}+ F_{n+2} + F_{n-2} + 2 G_n - G_{n+2} \right)(s-1) +2i \Pc \left(  2 H_{n-2}  - H_{n+2}\right)(s-1)  }{4\left[\frac{1}{s-1} + \mathcal{P}h_n(s-1)\right]},
\end{align*} 
 where $Re(s) > 1$. We also derive expressions for $b_n$ and $c_n$ in terms of the Mellin transform by using similar types of analysis as we performed for the above expression, and they are given as follows:
\begin{align*}
\mathcal{P}b_n(s) = \frac{2\Pc \left(   H_{n-2} + H_{n+2}\right)(s-1) - i \Pc \left( F_{n-2} - F_{n+2} + 2 G_n - G_{n+2} \right)(s-1)  }{4\left[\frac{1}{s-1} + \mathcal{P}h_n(s-1)\right]}, 
\end{align*}  
and
\begin{align*}
\mathcal{P}c_n(s) = \frac{\Pc \left(2F_{n}- F_{n-2} - F_{n+2} + 2 G_n + G_{n-2} + G_{n+2}\right)(s-1) +2i \Pc \left( H_{n+2}  - H_{n-2}\right)(s-1)  }{4\left[\frac{1}{s-1} + \mathcal{P}h_n(s-1)\right]}, 
\end{align*} 
for $Re(s) > 1$. This completes the proof of the Theorem \ref{th: Melin transforms for Fourier coefficients}.
\end{proof}
Now we use the  inversion formula for the Mellin transform to get following formulas for $a_n$, $b_n$, and $c_n$ for any $t > 1$:
   \begin{equation}\label{eq: Fourier coefficient of f_11}
       a_n(r) = \lim_{T\to\infty} \frac{1}{2\pi i} \int_{t-Ti}^{t+Ti} r^{-s}  \Ac_n(s)\,ds,
    \end{equation}
    
   \begin{equation}\label{eq: Fourier coefficient of f_22}
      b_n(r) = \lim_{T\to\infty} \frac{1}{2\pi i} \int_{t-Ti}^{t+Ti} r^{-s}  \Bc_n(s)\,ds,
    \end{equation}
  \begin{equation}\label{eq: Fourier coefficient of f_12}
   c_n(r) = \lim_{T\to\infty} \frac{1}{2\pi i} \int_{t-Ti}^{t+Ti} r^{-s}  \Cc_n(s)\,ds,
    \end{equation}  
    where
    \begin{align*}
        \Ac_n(s) = \frac{\Pc \left(2F_{n}+ F_{n+2} + F_{n-2} + 2 G_n - G_{n+2} \right)(s-1) +2i \Pc \left(  2 H_{n-2}  - H_{n+2}\right)(s-1)  }{4\left[\frac{1}{s-1} + \mathcal{P}h_n(s-1)\right]},
    \end{align*}
     \begin{align*}
        \Bc_n(s) = \frac{2\Pc \left(   H_{n-2} + H_{n+2}\right)(s-1) - i \Pc \left( F_{n-2} - F_{n+2} +  2 G_n - G_{n+2} \right)(s-1)  }{4\left[\frac{1}{s-1} + \mathcal{P}h_n(s-1)\right]},
    \end{align*}
     \begin{align*}
        \Cc_n(s) =  \frac{\Pc \left(2F_{n}- F_{n-2} - F_{n+2} + 2 G_n + G_{n-2} + G_{n+2}\right)(s-1) +2i \Pc \left( H_{n+2}  - H_{n-2}\right)(s-1)  }{4\left[\frac{1}{s-1} + \mathcal{P}h_n(s-1)\right]}.
    \end{align*}
This completes the recovery of the Fourier coefficients of each components of $2$-tensor field which will give the reconstruction for each components of $2$-tensor fields. Hence the proof of Theorem \ref{th:partial data recovery}, is completed. 


\section{Acknowledgements}\label{sec:acknowledgements}
The author acknowledges the support of UGC, the Government of India, with a research fellowship. This work was partially supported by the FIST program of the Department of Science and Technology, Government of India, Reference No. SR/FST/MS-I/2018/22(C). I would like to thank  Dr. Manmohan Vashisth and Dr. Rohit Kumr Mishra for suggesting the problem and useful discussions while working on this article.


\def\dbar{\leavevmode\hbox to 0pt{\hskip.2ex \accent"16\hss}d}
\begin{thebibliography}{ABKQ13}

\bibitem[ABKQ13]{Ambartsoumian2013}
G.~Ambartsoumian, J.~Boman, V.~P. Krishnan, and E.~T. Quinto.
\newblock Microlocal analysis of an ultrasound transform with circular source and receiver trajectories.
\newblock {\em American Mathematical Society, Series Contemporary Mathematics}, 598:45--58, 2013.

\bibitem[AJM24]{Gaik_Latifi_Rohit}
G.~Ambartsoumian, M.~J.~Latifi Jebelli, and R.~K. Mishra.
\newblock Numerical implementation of generalized {V}-line transforms on {2D} vector fields and their inversions.
\newblock {\em SIAM Journal on Imaging Sciences}, 17(1):595--631, 2024.

\bibitem[ALJM20]{Gaik_Mohammad_Rohit}
G.~Ambartsoumian, M.~J. Latifi~Jebelli, and R.~K. Mishra.
\newblock Generalized {V}-line transforms in {2D} vector tomography.
\newblock {\em Inverse Problems}, 36(10):104002, 2020.

\bibitem[AM13a]{Ambartsoumian_2013}
G.~Ambartsoumian and S.~Moon.
\newblock {A series formula for inversion of the {V}-line {R}adon transform in a disc}.
\newblock {\em Computers {\&} Mathematics with Applications}, 66(9):1567--1572, 2013.

\bibitem[AM13b]{Ambarsoumian_2013}
G.~Ambartsoumian and S.~Moon.
\newblock A series formula for inversion of the {V}-line {R}adon transform in a disc.
\newblock {\em Computers \& Mathematics with Applications}, 66(9):1567--1572, 2013.

\bibitem[Amb12]{Ambarsoumian_2012}
G.~Ambartsoumian.
\newblock Inversion of the {V}-line {R}adon transform in a disc and its applications in imaging.
\newblock {\em Computers \& Mathematics with Applications}, 64(3):260--265, 2012.

\bibitem[Amb23]{amb-book}
G.~Ambartsoumian.
\newblock {\em Generalized {R}adon Transforms and Imaging by Scattered Particles: Broken Rays, Cones, and Stars in Tomography}.
\newblock World Scientific, 2023.

\bibitem[AMZ24a]{Gaik_Rohit_Indrani}
G.~Ambartsoumian, R.~K. Mishra, and I.~Zamindar.
\newblock {V}-line 2-tensor tomography in the plane.
\newblock {\em Inverse Problems}, 40(3):Paper No. 035003, 24, 2024.

\bibitem[AMZ24b]{Indrani_2024v}
G.~Ambartsoumian, R.~K. Mishra, and I.~Zamindar.
\newblock V-line tensor tomography: numerical results.
\newblock {\em arXiv preprint arXiv:2405.03249}, 2024.

\bibitem[BMV25]{bhardwaj_2024}
R.~Bhardwaj, R.~K. Mishra, and M.~Vashisth.
\newblock On the inversion of generalized {V}-line transform of a vector field in $\mathbb{R}^2$.
\newblock {\em Mathematical Methods in the Applied Sciences}, 2025.

\bibitem[BZG96]{Basko1996AnalyticalRF}
R.~Basko, G.I. Zeng, and G.~T. Gullberg.
\newblock Analytical reconstruction formula for one-dimensional {C}ompton camera.
\newblock {\em 1996 IEEE Nuclear Science Symposium. Conference Record}, 3:1772--1776 vol.3, 1996.

\bibitem[BZG97a]{basko1997analytical}
R.~Basko, G.~L. Zeng, and G.~T. Gullberg.
\newblock Analytical reconstruction formula for one-dimensional {C}ompton camera.
\newblock {\em IEEE Transactions on Nuclear Science}, 44(3):1342--1346, 1997.

\bibitem[BZG97b]{basko1997fully}
R.~Basko, G.~L. Zeng, and G.~T. Gullberg.
\newblock Fully three-dimensional image reconstruction from "{V}"-projections acquired by {C}ompton camera with three vertex electronic collimation.
\newblock In {\em 1997 IEEE Nuclear Science Symposium Conference Record}, volume~2, pages 1077--1081. IEEE, 1997.

\bibitem[BZG98]{Basko1998ApplicationOS}
R.~Basko, G.~L. Zeng, and G.~T. Gullberg.
\newblock Application of spherical harmonics to image reconstruction for the {C}ompton camera.
\newblock {\em Physics in medicine and biology}, 43 4:887--94, 1998.

\bibitem[Den94]{Denisyuk_1994}
A.~Denisyuk.
\newblock Inversion of the generalized {R}adon transform.
\newblock {\em Translations of the American Mathematical Society-Series 2}, 162:19--32, 1994.

\bibitem[Den06]{Denisyuk_2006}
A.~Denisyuk.
\newblock Inversion of the {X}-ray transform for 3{D} symmetric tensor fields with sources on a curve.
\newblock {\em Inverse problems}, 22(2):399, 2006.

\bibitem[DS15]{Derevtsov_2015}
E.~Yu. Derevtsov and I.~E. Svetov.
\newblock Tomography of tensor fields in the plain.
\newblock {\em Eurasian J. Math. Comput. Appl}, 3(2):24--68, 2015.

\bibitem[Esk04]{Eskin}
G.~Eskin.
\newblock Inverse boundary value problems in domains with several obstacles.
\newblock {\em Inverse Problems}, 20.5:1497, 2004.

\bibitem[FMS08]{Florescu2008SinglescatteringOT}
L.~Florescu, V.~A. Markel, and J.~C. Schotland.
\newblock Single-scattering optical tomography: simultaneous reconstruction of scattering and absorption.
\newblock {\em Physical review. E, Statistical, nonlinear, and soft matter physics}, 81 1 Pt 2:016602, 2008.

\bibitem[FMS10a]{Florescu2010InversionFF}
L.~Florescu, V.~A. Markel, and J.~C. Schotland.
\newblock Inversion formulas for the broken-ray {R}adon transform.
\newblock {\em Inverse Problems}, 27:025002, 2010.

\bibitem[FMS10b]{florescu2010single}
L.~Florescu, V.~A. Markel, and J.~C. Schotland.
\newblock Single-scattering optical tomography: simultaneous reconstruction of scattering and absorption.
\newblock {\em Physical Review E}, 81(1):016602, 2010.

\bibitem[FMS11]{Florescu-Markel-Schotland}
L.~Florescu, V.~A. Markel, and J.~C. Schotland.
\newblock Inversion formulas for the broken-ray {R}adon transform.
\newblock {\em Inverse Problems}, 27(2):025002, 2011.

\bibitem[FMS18]{florescu_2018}
L.~Florescu, V.~A. Markel, and J.~C. Schotland.
\newblock Nonreciprocal broken ray transforms with applications to fluorescence imaging.
\newblock {\em Inverse Problems}, 34(9):094002, 2018.

\bibitem[FSM09]{florescu2009single}
L.~Florescu, J.~C. Schotland, and V.~A. Markel.
\newblock Single-scattering optical tomography.
\newblock {\em Physical Review E}, 79(3):036607, 2009.

\bibitem[HMS17]{Haltmeier_attenuated}
M.~Haltmeier, S.~Moon, and D.~Schiefeneder.
\newblock Inversion of the attenuated {V}-line transform with vertices on the circle.
\newblock {\em IEEE Trans. Comput. Imaging}, 3(4):853--863, 2017.

\bibitem[Hub14]{Mark}
M.~Hubenthal.
\newblock The broken ray transform on the square.
\newblock {\em J. Fourier Anal. Appl.}, 20(5):1050--1082, 2014.

\bibitem[Ilm13]{Ilmavirta}
J.~Ilmavirta.
\newblock Broken ray tomography in the disc.
\newblock {\em Inverse Problems}, 29(3):035008, 17, 2013.

\bibitem[IS16]{Ilmavirta_Mikko}
J.~Ilmavirta and M.~Salo.
\newblock Broken ray transform on a {R}iemann surface with a convex obstacle.
\newblock {\em Comm. Anal. Geom.}, 24(2):379--408, 2016.

\bibitem[JKR24]{Shubham}
S.~R. Jathar, M.~Kar, and J.~Railo.
\newblock Broken ray transform for twisted geodesics on surfaces with a reflecting obstacle.
\newblock {\em J. Geom. Anal.}, 34(7):Paper No. 212, 49, 2024.

\bibitem[Kat06]{Katsevich_2006}
A.~Katsevich.
\newblock Improved cone beam local tomography.
\newblock {\em Inverse Problems}, 22(2):627, 2006.

\bibitem[KS14]{Carlos_Mikko}
C.~Kenig and M.~Salo.
\newblock The {C}alderón problem with partial data on manifolds and applications.
\newblock {\em Analysis {\&} PDE}, 6.8:2003--2048, 2014.

\bibitem[MH17]{Haltmeier_2017}
S.~Moon and M.~Haltmeier.
\newblock Analytic inversion of a conical {R}adon transform arising in application of {C}ompton cameras on the cylinder.
\newblock {\em SIAM J. Imaging Sci.}, 10(2):535--557, 2017.

\bibitem[Mis20]{Mishra_2020}
R.~K. Mishra.
\newblock Full reconstruction of a vector field from restricted {D}oppler and first integral moment transforms in $\mathbb{R}^n$.
\newblock {\em Journal of Inverse and Ill-posed Problems}, 28(2):173--184, 2020.

\bibitem[Mon16]{Monard1}
F.~Monard.
\newblock Efficient tensor tomography in fan-beam coordinates.
\newblock {\em Inverse Probl. Imaging}, 10(2):433--459, 2016.

\bibitem[MPZ25]{Indrani_2024}
Rohit~Kumar Mishra, Anamika Purohit, and Indrani Zamindar.
\newblock Tensor tomography using {V}-line transforms with vertices restricted to a circle.
\newblock {\em Anal. Math. Phys.}, 15(1):Paper No. 23, 2025.

\bibitem[MS21]{Rohit_Suman_2021}
R.~K. Mishra and S.~K. Sahoo.
\newblock Injectivity and range description of integral moment transforms over $m$-tensor fields in $\mathbb{R}^n$.
\newblock {\em SIAM Journal on Mathematical Analysis}, 53(1):253--278, 2021.

\bibitem[Nat01]{Natterer_2001}
F.~Natterer.
\newblock {\em The mathematics of computerized tomography}.
\newblock SIAM, 2001.

\bibitem[NW01]{F.Natterer_2001}
F.~Natterer and F.~W{\"u}bbeling.
\newblock {\em Mathematical methods in image reconstruction}.
\newblock SIAM, 2001.

\bibitem[Pal09]{Palamodov_2009}
V.~Palamodov.
\newblock Reconstruction of a differential form from {D}oppler transform.
\newblock {\em SIAM journal on mathematical analysis}, 41(4):1713--1720, 2009.

\bibitem[PFD95]{Flajolet_1995}
X.~Gourdon P.~Flajolet and P.~Dumas.
\newblock Mellin transforms and asymptotics: Harmonic sums.
\newblock {\em Theoretical computer science}, 144(1-2):3--58, 1995.

\bibitem[Sha94]{Sharafutdinov_Book}
V.~A. Sharafutdinov.
\newblock {\em Integral geometry of tensor fields}.
\newblock Inverse and Ill-posed Problems Series. VSP, Utrecht, 1994.

\bibitem[She15]{Sherson}
B.~Sherson.
\newblock {\em Some Results in Single-Scattering Tomography}.
\newblock PhD thesis, Oregon State University, 2015.
\newblock {PhD} Advisor: D. Finch.

\bibitem[TKK18]{Terzioglu_2018}
F.~Terzioglu, P.~Kuchment, and L.~Kunyansky.
\newblock Compton camera imaging and the cone transform: a brief overview.
\newblock {\em Inverse Problems}, 34(5):054002, apr 2018.

\bibitem[TN11]{T.Truong_2011}
T.~T. Truong and M.~K. Nguyen.
\newblock On new {V}-line {R}adon transforms in {$\mathbb{R}^2$} and their inversion.
\newblock {\em J. Phys. A}, 44(7):075206, 13, 2011.

\bibitem[WO19]{walker2019broken}
M.~R. Walker and J.~A. O’Sullivan.
\newblock The broken ray transform: additional properties and new inversion formula.
\newblock {\em Inverse Problems}, 35(11):115003, 2019.

\end{thebibliography}


\addcontentsline{toc}{section}{References}
\end{document}